\documentclass{amsart}
\usepackage{amssymb, graphicx}
\usepackage[mathscr]{eucal}
\usepackage{slashbox}
\usepackage{amsmath,amscd}
\usepackage{rotating}

\theoremstyle{plain}
\newtheorem{prop}{Proposition}[section]
\newtheorem{coro}[prop]{Corollary}

\newtheorem{conj}[prop]{Conjecture}
\newtheorem{lemm}[prop]{Lemma}

\newtheorem{thm}[prop]{Theorem}

\theoremstyle{definition}
\newtheorem{defn}[prop]{Definition}

 \DeclareMathOperator{\Br}{Br}

\def\mcg#1;#2{\Gamma_{#1,#2}}
\def\fg#1;#2{\Pi_{#1,#2}}
\def\tb#1;#2{\mathscr{K}_{\frac{#1}{#2}}}

\begin{document}

\title[Further Study of Kanenobu Knots]
{Further Study of Kanenobu Knots}

\keywords{Crossing number, Kanenobu knots, Khovanov homology}
\author{Khaled Qazaqzeh}
\address{Department of Mathematics\\ Faculty of Science \\ Kuwait University\\
P. O. Box 5969\\ Safat-13060\\ Kuwait, State of Kuwait}
\email{khaled@sci.kuniv.edu.kw}

\author{Isra Mansour}

\address{Department of Mathematics\\
Faculty of Science \\ Yarmouk University\\ Irbed, Jordan}
\email{esramansour@yahoo.com}

\subjclass[2000]{57M27}


\date{15/04/2014}

\begin{abstract}
We determine the rational Khovanov bigraded homology groups of all
Kanenobu knots. Also, we determine the crossing number for all
Kanenobu knots $K(p,q)$ with $pq > 0$ or $|pq|\leq \max \{|p|,
|q|\}$. In the case where $pq < 0$ and $|pq| > \max \{|p|, |q|\}$,
we conjecture that the crossing number is $|p| + |q| + 8$.
\end{abstract}

\maketitle

\section{Introduction}

Shortly after the discovery of the Jones polynomial, the HOMFLY-PT,
and the Kauffman polynomial invariants, Kanenobu in \cite{Ka1}
introduced an infinite family of knots that consists of infinite
classes of knots that have the same HOMFLY-PT and Jones polynomials
which are hyperbolic, fibered, ribbon, of genus 2 and 3-bridge, but
with distinct Alexander module structures .

Kanenobu knots are denoted by $K(p, q)$ for two integers $p$ and
$q$, where $p$ and $q$ denote the number of half twists as given in
figure \ref{figure13}.

\begin{figure}[h]
\centering
\includegraphics[width=14cm,height=4cm]{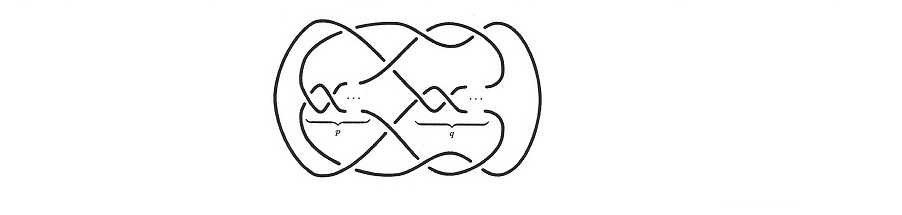}
\caption{The diagram of the Kanenobu knot $K(p,q)$ with $p,q> 0$}
\label{figure13}
\end{figure}

The motivation behind defining Kanenobu knots is the following
theorem:

\begin{thm}[Main Theorem, \cite{Ka1}]\label{main}
There exists infinitely many examples of infinitely many knots in
$\mathbb{S}^{3}$ with the same HOMFLY-PT polynomial invariant and,
therefore, the same Jones polynomial but distinct Alexander module
structures which are hyperbolic, fibered, ribbon, of genus 2 and
3-bridge.
\end{thm}

We summarize the main properties of Kanenobu knots that will be
needed in this paper as follows:

\begin{prop}[\cite{Ka1,Ka2}]\label{properties} For the above Kanenobu knots, we
have
\begin{enumerate}
\item $K(p_1, q_1) \approx K(p_2, q_2)$ iff $(p_1, q_1) = (p_2, q_2)$ or $(p_1, q_1) = (q_2, p_2)$.
\item $K'(p, q) \approx K(-p, -q)$, where $K'$ denotes the mirror image of $K$.
\item $K(p, q)$ is prime except for $K(0, 0)$, where $K(0, 0) \approx 4_1 \# 4_1$.
\item $V(p, q)  = (-t)^{p + q} (V(0, 0) - 1) + 1 = (-t)^{p + q} ((t^{-2} - t^{-1} + 1 - t + t^2)^2 - 1) + 1)$.
\end{enumerate}
\end{prop}



In this paper, we determine the rational Khovanov bigraded homology
groups of Kanenobu knots that first appeared in \cite{Ka1,Ka2} and
this confirms the result of Theorem 7 in \cite{GW}. Based on this
result, we show that all Kanenobu knots are homologically
thin. 
Also, we determine the crossing number for all Kanenobu knots
$K(p,q)$ with $pq > 0$ or $|pq|\leq \max \{|p|, |q|\}$ using the
inequality appears in [Theorem 1,\cite{Ki}]. Finally, we conjecture
that the crossing number for the Kanenobu knot $K(p,q)$ is $|p| +
|q| + 8$ with $pq < 0$ and $|pq|
> \max \{|p|, |q|\}$.



\section{Background and Notations}

In this section, we give the basic notations and tools for this
paper.

\subsection{The Jones Polynomial}









\begin{defn} 
The Jones polynomial of a link $L$ is an invariant of the
equivalence class of the oriented link $L$ that is characterized as
follows:

\begin{enumerate}
\item $V(O) = 1$, where $O$ is the unknot.
\item Let $L_+$, $L_-$, and $L_0$ be three oriented link
diagrams that are identical except at a small region as shown in
figure \ref{figure11}, then the Jones polynomial satisfies the
following skein relation:
\begin{align}
t^{-1}V(L_+) - t V(L_-) = (t^\frac{1}{2} - t^\frac{-1}{2})V(L_0).
\end{align}
\end{enumerate}

    \begin{figure}[h]
\centering
\includegraphics[width=7cm,height=2.2cm]{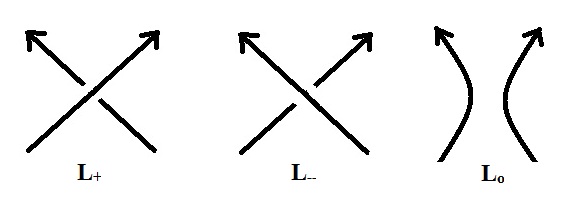}
\caption{$L_+$, $L_-$, and $L_0$, respectively.} \label{figure11}
\end{figure}
\end{defn}

We mention some facts regarding the Jones polynomial that will be
used later.

\begin{prop}
Let $L'$ be the mirror image of the link $L$, then $V_L(t)$ =
$V_{L'}(t^{-1})$.
\end{prop}


\begin{thm} [Theorem 1 \& 2 \& Corollary 1, \cite{T}]\label{T}
Let $D$ be a connected diagram of $n$ crossings of an oriented link
$L$, and $\Br(V)$ denotes the breadth of the Jones polynomial
$V_L(t)$ of a link $L$ that is the difference between the maximal
and the minimal exponents of $t$, then:
\begin{enumerate}
\item If $D$ is alternating and reduced, then $\Br(V) = n$.
\item If $D$ is non-alternating and prime, then $\Br(V) < n$.
\item If the link $L$ admits a reduced alternating diagram with $n$
crossings, then it can not be projected with fewer than $n$
crossings.
\end{enumerate}
\end{thm}


\subsection{The Kauffman Polynomial}






We first define the writhe of any link diagram before we define the
Kauffman polynomial.

\begin{defn}
The writhe of a diagram $D$ of an oriented link $L$, denoted by
$w(D)$, is the difference between the number of positive crossings
and the number of negative crossings of $D$, where positive and
negative crossings are defined according to figure \ref{figure1}.
\end{defn}

The Kauffman polynomial $F_L(a, x)$ is a two variable link
polynomial of an oriented link $L$ which is defined on a link
diagram $D$ of $L$ as follows:

\begin{align}
F_L(a, x) = a^{-w(D)} \Lambda_D(a, x),
\end{align}

where $\Lambda_D(a, x)$ is a polynomial in $a$ and $x$ defined on a
diagram $D$ of an unoriented link $L$ by the following properties:

\begin{enumerate}
\item $\Lambda(O)$ = 1, where $O$ is the unknot.
\item $\Lambda(s_r)$ = $a\Lambda(s)$, where $s_r$ is the strand $s$ with a positive curl.
\item $\Lambda(s_l)$ = $a^{-1}\Lambda(s)$, where $s_l$ is the strand $s$ with a negative curl.
\item $\Lambda$ is unchanged under Reidemeister moves of types II and III.
\item Let $L_-$, $L_+$, $L_\infty$ and $L_0$ are identical diagrams except at a small
region where they are as shown in figure \ref{figure8}, then
$\Lambda$ satisfies  Kauffman's skein relation:
\begin{align}
\Lambda(L_-) + \Lambda(L_+) = x (\Lambda(L_\infty) + \Lambda(L_0)).
\end{align}

\end{enumerate}

\begin{figure}[h]
\centering
\includegraphics[width=10cm,height=2.5cm]{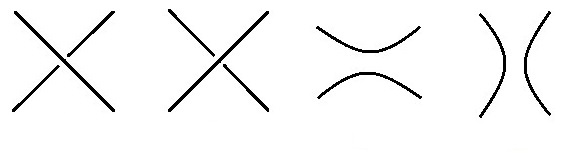}
\caption{$L_-$, $L_+$, $L_\infty$ and $L_0$, respectively.}
\label{figure8}
\end{figure}

For a link $L$, the $Q$ polynomial is related to the Kauffman
polynomial by the following relation:
\begin{align}
 Q_{L}(x) = F_{L}(1, x).
\end{align}

\begin{prop}[Property 2, \cite{BLM}]
The $Q$ polynomial satisfies the following properties:
\begin{enumerate}
    \item $Q(L_1 \# L_2)$ = $Q(L_1)$ $Q(L_2)$, where $L_1 \# L_2$ is the connected sum of $L_1$ and $L_2$.
    \item $Q(L')$ = $Q(L)$, where $L'$ is the mirror image of $L$.
\end{enumerate}
\end{prop}

\subsection{Khovanov Invariant}
We use the rational version of this link invariant that was first
introduced by Khovanov in \cite{K1} by setting $ c = 0$ and
tensoring with $\mathbb{Q}$.

For a diagram $D$ of the oriented link $L$, we denote the
isomorphism classes of $\mathcal{H}^{i,j}(D)$ by
$\mathcal{H}^{i,j}(L)$. We denote $\overline{\mathcal{H}}^{i}(D)$,
and $\mathcal{H}^{i}(D)$ to be the $i-$th homology group of the
complex $\overline{\mathcal{C}}(D)$, and $\mathcal{C}(D)$
respectively. Also, we denote $\overline{\mathcal{H}}^{i,j}(D)$, and
$\mathcal{H}^{i,j}(D)$ to be the $j-$th graded component of
$\overline{\mathcal{H}}^{i}(D)$, and $\mathcal{H}^{i}(D)$
respectively. Therefore, we have
\[
\overline{\mathcal{H}}^{i}(D) = \oplus_{j\in
\mathbb{Z}}\overline{\mathcal{H}}^{i,j}(D), \quad \text{and} \quad
\mathcal{H}^{i}(D) = \oplus_{j\in \mathbb{Z}}\mathcal{H}^{i,j}(D).
\]

The homology groups $\overline{\mathcal{H}}(D)$ and $\mathcal{H}(D)$
are related as follows:
\begin{equation}\label{relation}
\mathcal{H}^{i, j}(D) = \overline{\mathcal{H}}^{i + x(D), j + 2x(D)
- y(D)}(D),
\end{equation}
where $x(D)$ and $y(D)$ are the number of crossings of the form of
$L_{-}$ and $L_{+}$ respectively as shown in figure \ref{figure11}.


This invariant is a cohomology theory whose graded Euler
characteristic is equal to the normalized Jones polynomial given in
equation \ref{main}. Moreover, it is more powerful than the
normalized Jones polynomial in many cases.
\begin{equation}\label{main}
\sum_{i,j\in\mathbb{Z}}(-1)^{i}q^{j}\dim\mathcal{H}^{i,j}(L) =
(q^{-1} + q)V(L)_{\sqrt{t} = -q}.
\end{equation}

\subsection{Exact Sequence}
We denote $D(*0)$ and $D(*1)$ by the diagrams obtained by applying
the 0-resolution and the 1-resolution at some fixed crossing of the
link diagram $D$ as shown in the figure \ref{figure2}.

\begin{figure}[h]
  \centering
  \includegraphics[width=6cm]{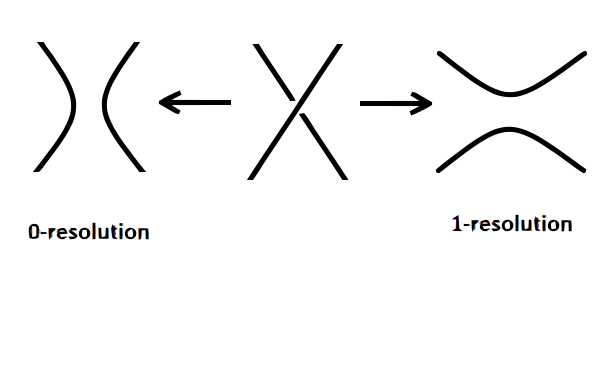}
\vspace{-1.5cm}
  \caption{The two resolutions of a link diagram at some crossing}\label{figure2}
\end{figure}

It is clear that $\mathcal{C}(D(*0))$ and
$\mathcal{C}(D(*1))[-1]\{-1\}$ are subcomplexes of $\mathcal{C}(D)$
and form a short exact sequence
\[
0 \rightarrow \mathcal{C}(D(*1))[-1]\{-1\} \rightarrow
\mathcal{C}(D) \rightarrow \mathcal{C}(D(*0)) \rightarrow 0,
\]
with degree preserving maps. Therefore, this induces a long exact
sequence on homology:
\begin{align}\label{sequence}
\cdots\rightarrow \overline{\mathcal{H}}^{i-1,j}(D(*0))
\stackrel{\delta}{\rightarrow}
\overline{\mathcal{H}}^{i-1,j-1}(D(*1))\rightarrow
\overline{\mathcal{H}}^{i,j}(D)\rightarrow
\overline{\mathcal{H}}^{i,j}(D(*0))\stackrel{\delta}{\rightarrow}
\overline{\mathcal{H}}^{i,j-1}(D(*1))\rightarrow \cdots
\end{align}

\subsection{Lee Invariant}
The Lee invariant is a variant of the rational Khovanov homology
obtained from the same underlying complex, but with different
differential. We let $H^{i}_{Lee}(L)$ to denote the $i-$th homology
group of the complex $\mathcal{C}(D)$ using Lee's differential. The
results needed from Lee's invariant in this paper is summarized by
\begin{prop}[Proposition 3.3, \cite{L}]\label{lee}
Let $L$ be an oriented link with $k$ components $L_{1},L_{2}, \ldots
,L_{k}$. Then
\begin{enumerate}
\item $\dim(H_{Lee}(L)) = 2^{k}$.
\item For every orientation $\theta$ of L there is a generator of homology in degree
\[
2 \times \sum_{l\in E, m\in \bar{E}}lk(L_{l},L_{m})
\]
where $E \subset \{1, 2, \ldots , k\}$ indexes the set of components
of $L$ whose original orientation needs to be reversed to get the
orientation $\theta$ and $\bar{E} = \{1,\ldots, k\}\setminus E$. The
linking numbers $lk(L_{l},L_{m})$ are the linking number (for the
original orientation) between component $L_{l}$ and $L_{m}$.
\item There is a spectral sequence converging to $H_{Lee}(L)$ with
$E_{1}^{s,t} = \mathcal{H}^{s+t,4s}(L)$.
\end{enumerate}
\end{prop}

\subsection{Rasmussen $s$-invariant}
In \cite{R} Rasmussen defines the $s$-invariant as the average of
the $j$-gradings of the two copies of $\mathbb{Q}$ in the
$H^{i}_{Lee}(K)$ which is the $E_{\infty}$ page of the above
spectral sequence for the knot $K$. He shows that the two
$j$-gradings differ by 2. Moreover, he shows that this invariant
satisfies the following inequality:
\begin{equation}\label{s}
\left|s(K)\right| \leq 2g^{*}(K),
\end{equation}
where $g^{*}(K)$ is the slice genus of the knot $K$.


\section{The crossing number of Kanenobu knots}


In this section, we establish the first goal of this paper that is
evaluating the crossing number of Kanenobu knots $K(p,q)$ with $pq >
0$ or $|pq|\leq \max \{|p|, |q|\}$. In the case when $pq < 0$ and
$|pq| > \max \{|p|, |q|\}$, we conjecture that the crossing number
is $|p| + |q| + 8$. But first, we state and prove some results
needed for the proof of the main theorem.

\begin{prop}
Let $\Br(V(p, q))$ denote the breadth of the Jones polynomial of
Kanenobu knot $K(p, q)$. Then,
\[
\Br(V(p, q)) =
  \begin{cases}
    { |p| + |q| + 4}, & \text{if $|p + q|  > 4$},
    \\
  {8}, & \text{if $|p + q| \leq 4$}.
  \end{cases}
\]
\end{prop}

\begin{proof}
The Jones polynomial of the Kanenobu knot $V(p, q)$ is:
\begin{align*}
V(p, q)  = & (-t)^{p + q} (V(0, 0) -1) + 1\\
 = & (-t)^{p +q} (t^{-4}-2t^{-3}+ 3t^{-2}-4t^{-1}+ 4- 4t + 3t^2 - 2t^3 + t^4) + 1
\\  = & (-1)^{p+q}(t^{p+q-4}-2t^{p+q-3}+ 3t^{p+q-2}-4t^{p+q-1}+ 4t^{p+q}-4t^{p+q+1}+ \\ & 3t^{p+q+2}
- 2t^{p+q+3}+ t^{p+q+4}) + 1.
\end{align*}

We have three cases to consider
\begin{enumerate}

\item If $p + q < - 4$, then the highest and lowest exponents of $t$ in $V(p, q)$ are 0 and $p +
q - 4$, respectively. So $\Br(V(p, q)) =  - p - q + 4$.

\item If $|p + q| \leq 4$, then the highest
and lowest exponents of $t$ in $V(p, q)$ are $p + q + 4$ and $p + q
- 4$, respectively. So $\Br(V(p, q)) = 8$.

\item If $p + q > 4$, then the highest
and lowest exponents of $t$ in $V(p, q)$ are $p + q + 4$ and 0,
respectively. So $\Br(V(p, q)) = p + q + 4$.
\end{enumerate}
\end{proof}

The $Q$ polynomial of any Kanenobu knot is given by the following
proposition:

\begin{prop}[Proposition 4.5, \cite{Ka2}]
Let $Q(a,b)$ be the $Q$ polynomial of the Kanenobu knot $K(a,b)$,
then we have
\begin{align*}
Q(a,b) = -\sigma_{a}\sigma_{b}(Q(8_{9})-1) +
x^{-1}(\sigma_{a+1}\sigma_{b+1} +
\sigma_{a-1}\sigma_{b-1})(Q(8_{8})-1) + 1,
\end{align*}
where $Q(8_{8}) = 1 + 4x + 6x^{2} - 10x^{3} - 14x^{4} +4x^{5}
+8x^{6} +2x^{7}$ and $Q(8_{9}) = -7 + 4x + 16x^{2} - 10x^{3} -
16x^{4} +4x^{5} +8x^{6} +2x^{7}$.
\end{prop}

In the above proposition, $\sigma_{n}$ is defined as follows:

\[ \sigma_{n} =
  \begin{cases}
  \frac{\alpha^{n} - \beta^{n}}{\alpha - \beta}, & \text{if}\  n > 0,
     \\
        0, & \text{if} \ n = 0,
     \\
  -\frac{\alpha^{-n} - \beta^{-n}}{\alpha - \beta}, & \text{if} \ n < 0,
  \end{cases}
\]
where $\alpha + \beta = x$ and $\alpha\beta = 1$.

The degree of the $Q$ polynomial of any Kanenobu knot is given in
the following proposition:

\begin{prop}\label{degree}

For the Kanenobu knot $K(p, q)$, we have
\[ \deg Q(p, q) =
  \begin{cases}
  |p|+|q|+6, & \text{if} \ pq \ \geq 0,
     \\
        |p|+|q|+5, & \text{otherwise},
  \end{cases}
\]
\end{prop}
\begin{proof}
We claim that $\sigma_{n} = \frac{n}{|n|}S_{|n|-1}(x)$, where
$S_{k}(x)$ is $k-th$ Chebyshev polynomial of the first kind which is
defined inductively by $S_{-1}(x) = 0, S_{0}(x) = 1$ and $S_{k}(x) =
xS_{k-1}(x) - S_{k-2}(x)$. It is clear that this claim is true for $
n = 0,1$. Now we prove this claim by showing that $(\alpha +
\beta)\sigma_{k} = \sigma_{k+1} + \sigma_{k-1}$ for $k \geq 1$. We
prove the last statement as follows:
\begin{align*}
(\alpha + \beta)\sigma_{k} = & (\alpha + \beta)(\alpha^{k-1} +
\alpha^{k-2}\beta + \ldots + \alpha\beta^{k-2} + \beta^{k-1})\\
= & \alpha(\alpha^{k-1} + \alpha^{k-2}\beta + \ldots +
\alpha\beta^{k-2} + \beta^{k-1}) + \\
& \beta(\alpha^{k-1} + \alpha^{k-2}\beta + \ldots +
\alpha\beta^{k-2} + \beta^{k-1}) \\
 = & (\alpha^{k} + \alpha^{k-1}\beta + \ldots + \alpha^{2}\beta^{k-2}
+ \alpha\beta^{k-1}) + \\
& (\alpha^{k-1}\beta + \alpha^{k-2}\beta^{2} + \ldots +
\alpha\beta^{k-1} + \beta^{k})\\
= &  (\alpha^{k} + \alpha^{k-1}\beta + \ldots +
\alpha^{2}\beta^{k-2} + \alpha\beta^{k-1} + \beta^{k}) + \\
& (\alpha^{k-1}\beta + \alpha^{k-2}\beta^{2} + \ldots +
\alpha\beta^{k-1}) \\
= & \sigma_{k+1} + (\alpha^{k-2} + \alpha^{k-3}\beta + \ldots +
\beta^{k-2}) = \sigma_{k+1} + \sigma_{k-1},
\end{align*}
where we used the fact that $\alpha\beta = 1$ in the equation before
the last one. Now the claim follows since we have that $\sigma_{n} =
-\sigma_{-n}$ for $n < 0$.

Finally, the result follows since we have $\deg Q(p,q) = \deg
(\sigma_{|p| +1}\sigma_{|q|+1}) + \deg Q(8_{8}) - 1 = |p| + |q| + 7
- 1 = |p| + |q| + 6$ for $pq \geq 0$ and $\deg Q(p,q) = \deg
(\sigma_{|p + 1| }\sigma_{|q+1|}) + \deg Q(8_{8}) - 1 = |p| + |q| +
7 - 1 - 1 = |p| + |q| + 5$ for $pq < 0$.

In the above argument we used the that $\deg 0 = -1$.
\end{proof}


The main tool in proving our main theorem in this section is the
following:

\begin{thm} [Main Theorem, \cite{Ki}]\label{Ki}
Let $D$ be a diagram of a link $L$. Then,
\begin{align}
\deg Q \leq c(D) - b(D)
\end{align}
where $c(D)$ is the crossing number of the diagram $D$ of $L$,
$b(D)$ is the maximal bridge length of $D$, i.e., the maximal number
of consecutive over-pass or under-passes crossings and $\deg Q$ is
the degree of the $Q$ polynomial of the link $L$. Moreover, if $D$
is alternating (i.e. $b(D)$ = 1) and prime, then equality holds.
\end{thm}

\begin{thm}
The crossing number of the Kanenobu knot $K(p, q)$ is given as
follows:
\[ c(K(p, q)) =
  \begin{cases}
  |p| + |q| + 6, & \text{if} \ pq < 0 \ \text{and} \ |p| + |q| = 2,
   \\
 |p| + |q| + 7, & \text{if} \ \left(pq < 0 \ \text{and} \ (|p| = 1 \ \text{or} \ |q| = 1 \ \text{but not both}) \right) \\
  & \text{or} \ \left(pq = 0 \ \text{and} \ |p| + |q| = 1\right),
    \\
 |p| + |q| + 8, & \text{if} \ pq > 0 \ \text{or} \ \left(pq = 0 \ \text{and} \ |p| + |q| \neq
 1\right).
  \end{cases}
\]
\end{thm}

\begin{proof}
We give the proof case by case
\begin{enumerate}
\item If $pq < 0$ and  $|p| + |q| = 2$, then we have
$(p,q) = (1,-1)$ or $(p,q) = (-1,1)$. We have used Reidemeister
moves to construct a reduced alternating diagram of $K(1, -1)$ of 8
crossings as shown in figure \ref{figure18}. So by the third part of
Theorem \ref{T} we obtain $c(K(1, -1)) = 8 = |1| + |-1| + 6 $. Now
$c(K(-1, 1)) = c(K(1, -1)) = 8$ as a result of Proposition
\ref{properties}.
\item We have two subcases that will be treated separately
\begin{enumerate}
\item If $pq < 0$ and $(|p| = 1$ or $|q| = 1$ but not both), then we
have $(p,q) = (m,-1)$ or $(p,q) = (-1,m)$ or $(p,q) = (-m,1)$ or
$(p,q) = (1,-m)$ for a positive integer $m$. We have used
Reidemeister moves to construct a diagram of $K(m, -1)$ with $m + 8$
crossings as shown in figure \ref{figure19}. Now by Theorem \ref{Ki}
and Proposition \ref{degree}, we have

\[ \deg Q(m, -1)  = m + 1 + 5 = m + 6 \leq c(D) - b(D), \]
which implies that $c(D) \geq m + 6 + b(D)$ for any knot diagram $D$
of $K(m, -1)$.  We have $b(D) \geq 2$ for any knot diagram $D$ of
the knot $K(p, -1)$ since this knot is not alternating. i, e $b(D)
\geq 2$. This follows since if it is alternating, then by Theorem
\ref{Ki}, we would have $c(K(m,-1)) = \deg Q(m,-1) + 1 = m + 7 =
\Br(V(m,-1)) = 8$ or $c(K(m,-1)) = \deg Q(m,-1) + 1 = m + 7 =
\Br(V(m,-1)) = m + 1 + 4$ which is impossible in both cases. Now the
result follows since we have already constructed a diagram of $m +
8$ crossings with $b(D) =2$. Now $c(K(m, -1)) = c(K(-m, 1)) =
c(K(1,m)) = c(K(1,-m)) = m + 1 + 7 = m + 8$ as a result of
Proposition \ref{properties}.

\item If $pq = 0 \ \text{and} \ |p| + |q| = 1$, then $(p,q) = (1,0)$
or $(p,q) = (-1,0)$ or $(p,q) = (0,1)$ or $(p,q) = (0,-1)$. It is
enough to consider one case of these four cases as a result of
Proposition \ref{properties}. We have used Reidemeister moves to
construct a reduced alternating diagram of $K(1, 0)$ of 8 crossings
as shown in figure \ref{figure17}. So by the third part of Theorem
\ref{T} we obtain $c(K(1,0)) = 8 = |1| + |0| + 7$. Now $c(K(0, 1)) =
c(K(-1, 0)) = c(K(0, -1)) = c(K(1,0)) = 8 = |1| + |0| + 7$ as a
result of Proposition \ref{properties}.

\end{enumerate}
\item We have two also two subcases that will be treated separately
\begin{enumerate}
\item If $pq > 0$, then both of $p$ and $q$ are positive or negative
integers. Now by Theorem \ref{Ki} and Proposition \ref{degree}, we
have

\[ \deg Q(p, q)  = |p| + |q| + 6 \leq c(D) - b(D), \]

which implies that $c(D) \geq |p| + |q| + 6 + b(D)$ for any knot
diagram $D$ of $K(p, q)$. We have $b(D) \geq 2$ for any knot diagram
$D$ of the knot $K(p, q)$ since this knot is not alternating. This
follows since if it is alternating, then by Theorem \ref{Ki}, we
would have $c(K(p,q)) = \deg Q(p,q) + 1 = |p| + |q| + 7 =
\Br(V(p,q)) = 8$ or $c(K(p,q)) = \deg Q(p,q) + 1 = |p| + |q| + 7 =
\Br(V(p,q)) = |p| + |q| + 4$ which is impossible in both cases. Now
the result follows since we have already a diagram of $|p| + |q| +
8$ crossings with $b(D) =2$.

\item If $pq = 0$ and $|p| + |q| \neq 1$, then we have $(p,q) =
(m,0)$ or $(p,q) = (0,m)$  for $m \neq \pm 1$. Now by Theorem
\ref{Ki} and Proposition \ref{degree}, we have

\[ \deg Q(m, 0)  = |m| + 6 \leq c(D) - b(D), \] which implies that $c(D)
\geq |m| + 6 + b(D)$ for any knot diagram $D$ of $K(m, 0)$. We have
$b(D) \geq 2$ for any knot diagram $D$ of the knot $K(m, 0)$ since
this knot is not alternating. This follows since if it is
alternating, then by Theorem \ref{Ki}, we would have $c(K(m,0)) =
\deg Q(m,0) + 1 = |m| + 7 = \Br(V(m,0)) = 8$ or $c(K(m,0)) = \deg
Q(m,0) + 1 = |m| + 7 = \Br(V(m,0)) = |m| + 4$ which is impossible in
both cases. Now the result follows since we have already a diagram
of $|m| + 8$ crossings with $b(D) =2$. Now $c(K(m, 0)) = c(K(0,m)) =
|m| + 8$ as a result of Proposition \ref{properties}.

\end{enumerate}
\end{enumerate}
\end{proof}

\begin{coro}
The Kanenobu knots $K(p, q)$ are not alternating except for $K(0,
0),\\ K(1, 0), K(0, 1), K(-1, 0), K(0, -1), K(1, -1)$ and $K(-1,
1)$.
\end{coro}

\begin{figure}[p]
\centering
\includegraphics[width=14cm,height=21cm]{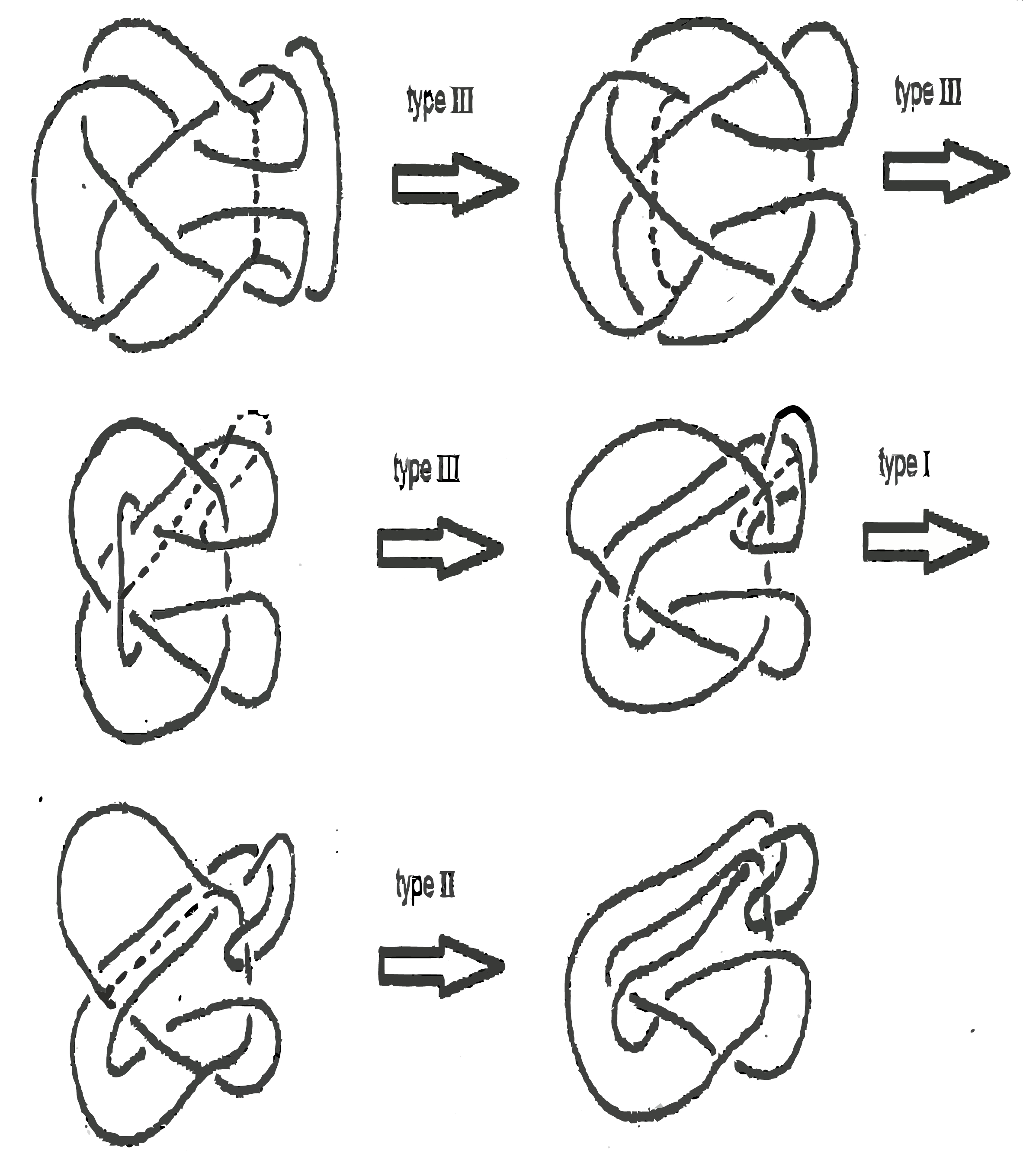}
\caption{Using Reidemeister moves to construct a reduced alternating
diagram of $K(1, 0)$.} \label{figure17}
\end{figure}

\begin{figure}[p]
\centering
\includegraphics[width=16cm,height=21cm]{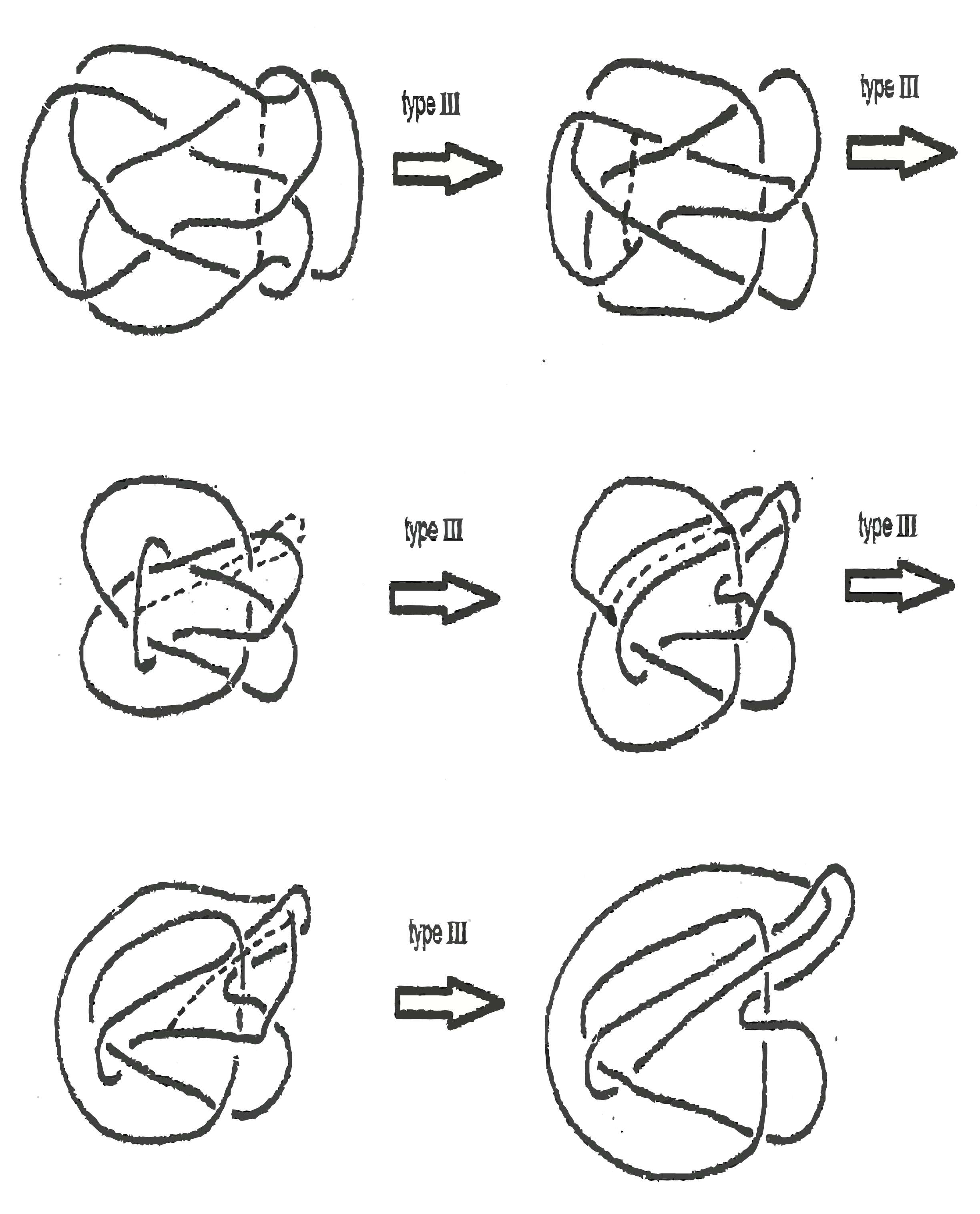}
\caption{Using Reidemeister moves to construct a reduced alternating
diagram of $K(1, -1)$.} \label{figure18}
\end{figure}

\begin{figure}[p]
\centering
\includegraphics[width=14cm,height=22cm]{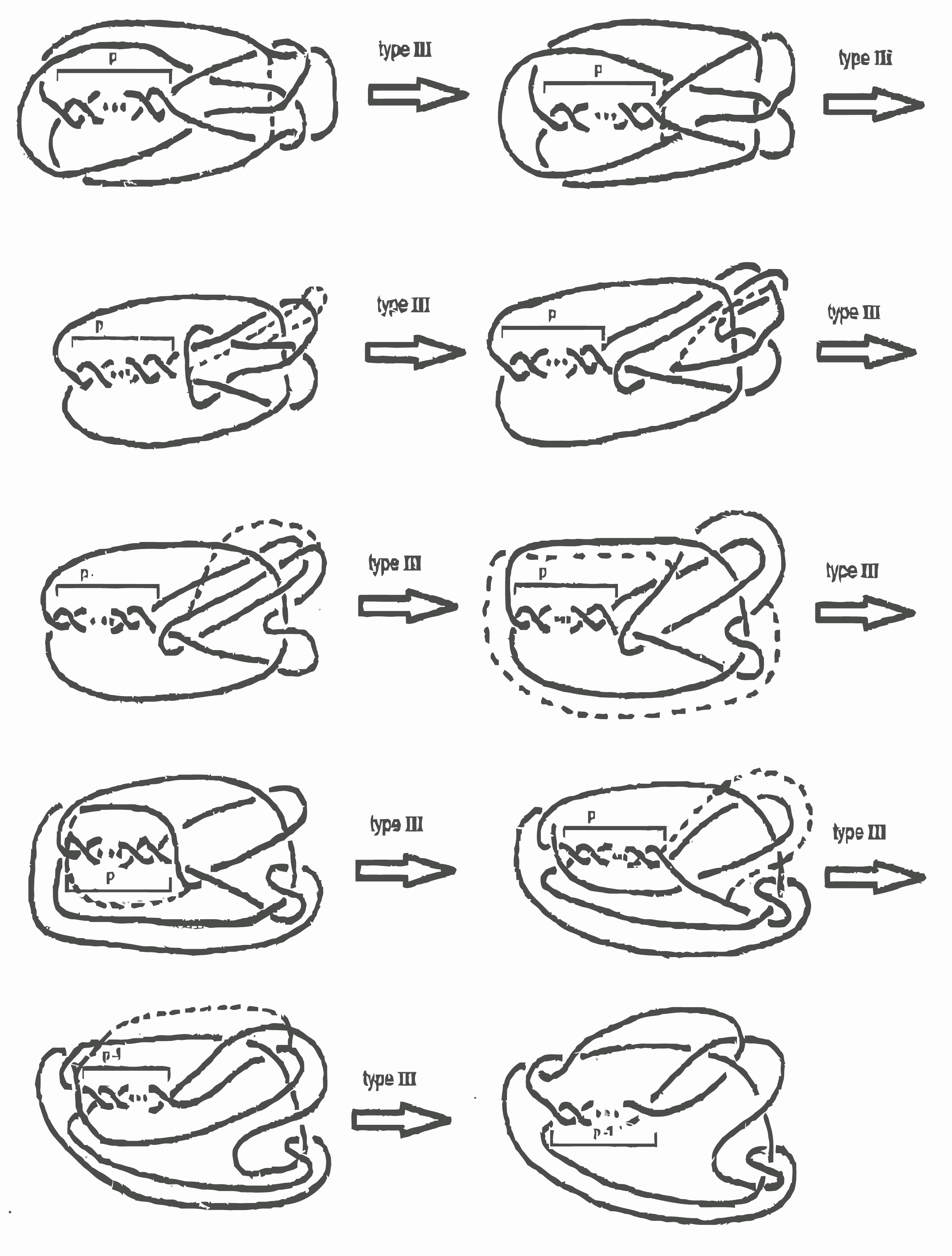}
\caption{Using Reidemeister moves to construct a reduced diagram of
$K(p, -1)$.} \label{figure19}
\end{figure}

\begin{prop}
For the case when $pq < 0$ with $|p| \neq 1$ and $|q| \neq 1$, we
have $|p|+|q| + 7 \leq c(K(p,q)) \leq |p|+|q| + 8$.
\end{prop}

\begin{proof}
We have $\deg Q(p,q) = |p| + |q| + 5$ with a knot diagram of $|p|+
|q| + 8$ crossings. We also know that the knot $K(p,q)$ in this case
is not alternating. This follows since if it is alternating, then by
Theorem \ref{Ki}, we would have $c(K(p,q)) = \deg Q(p,q) + 1 = |p| +
|q| + 6 = \Br(V(p,q)) = 8$ or $c(K(p,q)) = \deg Q(p,q) + 1 = |p| +
|q| + 6 = \Br(V(p,q)) = |p| + |q| + 4$ which is impossible in both
cases.

\end{proof}

We suggest the following conjecture that computes the crossing
number of the Kanenobu knot $K(p,q)$ in the above case.
\begin{conj}
For the case when $pq < 0$ with $|p| \neq 1$ and $|q| \neq 1$, we
have $c(K(p,q)) \leq |p|+|q| + 8$.
\end{conj}

\section{Khovanov homology of Kanenobu knots}

\begin{lemm}
The Khovanov bigraded homology groups of the figure eight knot
$\mathcal{H}^{i,j}(K)$ are given as follows:
\[
\dim \mathcal{H}^{i,j}(K) =
\begin{cases}
1, & \text{if} \ {(i,j) = (-2,-5),
(-1,-1),(0,-1),(0,1),(1,1),(2,5)},
      \\
0, & \text{otherwise}.
     \\
\end{cases}
\]

\end{lemm}

\begin{lemm}\label{slice}
The Kanenobu knot $K(p, 0)$ is a slice knot. Therefore, we obtain
$s(K(p, 0)) = 0$.
\end{lemm}
\begin{proof}
We know that the knot $K(p,0)$ is a ribbon knot by Theorem
\ref{main}, so it is a slice knot. The second part follows by
inequality \ref{s}.
\end{proof}

\begin{lemm}\label{seq}
We have
\begin{align*}
\dim \overline{\mathcal{H}}^{x(D), 2x(D) - y(D) - 1}(D) & =
\dim \overline{\mathcal{H}}^{x(D) + 1, 2x(D) - y(D) + 3}(D) + 1,\\
\dim \overline{\mathcal{H}}^{x(D), 2x(D) - y(D) + 1}(D) & = \dim
\overline{\mathcal{H}}^{ x(D) - 1, 2x(D) - y(D) - 3}(D) + 1, \\
\dim \mathcal{H}^{i,j}(K) & = \dim \mathcal{H}^{ i+1, j + 4}(K) + 1.
\end{align*}
for any diagram $D$ of a homologically thin knot $K$ with $s(K) = 0$
and $j-2i = \sigma(K) - 1$.
\end{lemm}

The following lemma will be used later in this paper.
\begin{lemm}\label{signaturel}
$\sigma(K(p, 0)) = 0 $ for any integer $p$.
\end{lemm}

\begin{prop}
The Khovanov bigraded homology groups of the disjoint union of two
figure eight knots $K \sqcup  K$ are given by the following:

\[
\dim \mathcal{H}^{i,j}(K\sqcup K) =
\begin{cases}
1,  & \text{if}  \ {(i,j) = (-4,-10),
(-2,-2),(0,-2),(0,2),(2,2),(4,10)},
      \\
    2,
     & \text{if}  \ {(i,j)  = (-3,-6), (-2,-6), (-2,-4), (-1,-4), (-1,-2)}, \\ & { \qquad \qquad \ (-1,0),
       (1,0), (1,2), (1,4), (2,4), (2,6), (3,6)},
     \\
6, &  \text{if}  \  {(i,j) = (0, 0)},
      \\
0, & \text{otherwise}.
     \\
\end{cases}
\]
\end{prop}

\begin{proof}
We use $\mathcal{H}^{i,j}(K\sqcup K) = {\bigoplus}_{l,m \in
\mathbb{Z}} (\mathcal{H}^{l,m}(K) \otimes \mathcal{H}^{i-l,j-m}(K))
$ given in [Corollary 12, \cite{K1}].
\end{proof}
\begin{prop}
The Khovanov bigraded homology groups of the Kanenobu knot $K(0,0)$
are given by the following:

\[
\dim \mathcal{H}^{i,j}(K(0,0)) =
\begin{cases}
1, & \text{if} \ {(i,j) = (-4,-9), (-3,-7),(-3,-5),(-2,-3),(2,3)},
\\ & \qquad \qquad \ (3,5), (3,7), (4,9),
      \\
2, & \text{if} \ {(i,j) = (-2,-5), (-1,-3), (-1,-1), (1,1), (1,3),
(2,5)},
\\
3, & \text{if} \  {(i,j) = (0, -1), (0,1)},
      \\
0, & \text{otherwise}.
     \\
\end{cases}
\]
\end{prop}

\begin{proof}
It is well-know that the Kanenobu knot $K(0,0)$ is the connected sum
of two figure eight knots $K\# K$. We use the following long exact
sequence given in [Proposition 34, \cite{K1}]:
\[ \ldots \rightarrow  \mathcal{H}^{i-1,j-1}(K\sqcup K) \rightarrow \mathcal{H}^{i-1,j-2}(K\# K)
\rightarrow \mathcal{H}^{i,j}(K\# K) \rightarrow
\mathcal{H}^{i,j-1}(K\sqcup K) \rightarrow \ldots
\]
We consider cases by case
\begin{enumerate}
\item If $j = 9$ then we have two subcases
\begin{enumerate}

\item If $i = 4$
\begin{align*}
\ldots \rightarrow \mathcal{H}^{4,11}(K\# K) \rightarrow
\mathcal{H}^{4,10}(K\sqcup K) \rightarrow & \mathcal{H}^{4,9}(K\# K)
\rightarrow \mathcal{H}^{5,11}(K\# K)\rightarrow \ldots \\
0 \rightarrow \mathbb{Q} \rightarrow & \mathcal{H}^{4,9}(K\# K)
\rightarrow 0, \\
\end{align*}
so we get $\dim \mathcal{H}^{4,9}(K\# K) = 1$.
\item If $i = 5$
\begin{align*} \ldots \rightarrow \mathcal{H}^{5,11}(K\# K)
\rightarrow \mathcal{H}^{5,10}(K\sqcup K) \rightarrow &
\mathcal{H}^{5,9}(K\# K)
\rightarrow \mathcal{H}^{6,11}(K\# K)\rightarrow \ldots \\
0 \rightarrow 0 \rightarrow & \mathcal{H}^{5,9}(K\# K)
\rightarrow 0, \\
\end{align*}
so we get $\dim \mathcal{H}^{5,9}(K\# K) = 0$.
\end{enumerate}
\item If $j = 7$ then we have two subcases
\begin{enumerate}

\item If $i = 3$
\begin{align*} \ldots \rightarrow \mathcal{H}^{3,8}(K\sqcup K)
\rightarrow \mathcal{H}^{3,7}(K\# K) \rightarrow &
\mathcal{H}^{4,9}(K\# K)
\rightarrow \mathcal{H}^{4,8}(K\sqcup K)\rightarrow \ldots \\
0 \rightarrow \mathcal{H}^{3,7}(K\# K) \rightarrow &
\mathcal{H}^{4,9}(K\# K) \rightarrow 0, \\
\end{align*}
so we get $\dim \mathcal{H}^{3,7}(K\# K) = \dim
\mathcal{H}^{4,9}(K\# K) = 1$.
\item If $i = 4$
\begin{align*} \ldots \rightarrow \mathcal{H}^{4,8}(K\sqcup K)
\rightarrow \mathcal{H}^{4,7}(K\# K) \rightarrow &
\mathcal{H}^{5,9}(K\# K)
\rightarrow \mathcal{H}^{5,8}(K\# K)\rightarrow \ldots \\
0 \rightarrow \mathcal{H}^{4,7}(K\# K)  \rightarrow &
\mathcal{H}^{5,9}(K\# K) \rightarrow 0, \\
\end{align*}
so we get $\dim \mathcal{H}^{4,7}(K\# K) =\dim \mathcal{H}^{5,9}(K\#
K) = 0$.
\end{enumerate}
\item If $j =5$, then we have

\begin{align*}
\ldots \rightarrow \mathcal{H}^{2,7}(K\# K) \rightarrow
\mathcal{H}^{2,6}(K\sqcup K) \rightarrow & \mathcal{H}^{2,5}(K\# K)
\rightarrow \mathcal{H}^{3,7}(K\# K) \\ \rightarrow
\mathcal{H}^{3,6}(K\sqcup K) & \rightarrow \mathcal{H}^{3,5}(K\# K)
\rightarrow \mathcal{H}^{4,7}(K\# K) \rightarrow \ldots \\
0 \rightarrow \mathbb{Q}\oplus\mathbb{Q} \rightarrow
\mathcal{H}^{2,5}(K\# K) \rightarrow & \mathbb{Q} \rightarrow
\mathbb{Q}\oplus\mathbb{Q}
\rightarrow \mathcal{H}^{3,5}(K\# K) \rightarrow 0, \\
\end{align*}
Now we have $\dim \mathcal{H}^{3,5}(K\# K) = \dim
\mathcal{H}^{4,9}(K\# K) = 1$ as a result of the third statement of
Lemma \ref{seq}. Therefore, we obtain $\dim \mathcal{H}^{2,5}(K\# K)
= 2$ since the Euler characteristic of an exact sequence must
vanish.
\item If $j = 3$, then we have
\begin{align*}
\ldots \rightarrow \mathcal{H}^{1,5}(K\# K) \rightarrow
\mathcal{H}^{1,4}(K\sqcup K) \rightarrow & \mathcal{H}^{1,3}(K\# K)
\rightarrow \mathcal{H}^{2,5}(K\# K) \\  \rightarrow
\mathcal{H}^{2,4}(K\sqcup K) \rightarrow \mathcal{H}^{2,3}(K\# K)
\rightarrow & \mathcal{H}^{3,5}(K\# K) \rightarrow \mathcal{H}^{3,4}(K\sqcup K) \rightarrow \ldots \\
0 \rightarrow \mathbb{Q}\oplus\mathbb{Q} \rightarrow
\mathcal{H}^{1,3}(K\# K) \rightarrow \mathbb{Q} \oplus \mathbb{Q}
\rightarrow &  \mathbb{Q} \oplus \mathbb{Q} \rightarrow
\mathcal{H}^{2,3}(K\# K) \rightarrow \mathbb{Q} \rightarrow 0, \\
\end{align*}
Now we have $\dim \mathcal{H}^{2,3}(K\# K) = \dim
\mathcal{H}^{3,7}(K\# K) = 1$ as a result of the third statement of
Lemma \ref{seq}. Therefore, we obtain $\dim \mathcal{H}^{1,3}(K\# K)
= 2$ since the Euler characteristic of an exact sequence must
vanish.
\item If $j = 1$, then we have
\begin{align*}
\ldots \rightarrow \mathcal{H}^{0,3}(K\# K) \rightarrow
\mathcal{H}^{0,2}(K\sqcup K) \rightarrow & \mathcal{H}^{0,1}(K\# K)
\rightarrow \mathcal{H}^{1,3}(K\# K) \\  \rightarrow
\mathcal{H}^{1,2}(K\sqcup K) \rightarrow \mathcal{H}^{1,1}(K\# K)
\rightarrow & \mathcal{H}^{2,3}(K\# K) \rightarrow
\mathcal{H}^{2,2}(K\sqcup K)
\rightarrow \mathcal{H}^{2,1}(K\# K) \rightarrow \ldots \\
0 \rightarrow \mathbb{Q} \rightarrow \mathcal{H}^{0,1}(K\# K)
\rightarrow \mathbb{Q} \oplus \mathbb{Q} \rightarrow & \mathbb{Q}
\oplus \mathbb{Q} \rightarrow \mathcal{H}^{1,1}(K\# K) \rightarrow
\mathcal{H}^{2,3}(K\# K) \rightarrow \mathbb{Q} \rightarrow 0, \\
\end{align*}
Now we have $\dim \mathcal{H}^{1,1}(K\# K) = \dim
\mathcal{H}^{2,5}(K\# K) = 2$ as a result of the third statement of
Lemma \ref{seq}. Also, we have $\dim \mathcal{H}^{0,1}(K\# K) \dim
\mathcal{H}^{-1,-3}(K\# K) + 1 = \dim \mathcal{H}^{1,3}(K\# K)+ 1 =
3$. The last isomorphism follows as a result of Theorem
\ref{properties} and [Corollary 11,\cite{K1}]. Therefore, we obtain
$\dim \mathcal{H}^{2,3}(K\# K) = 1$ since the Euler characteristic
of an exact sequence must vanish.
\end{enumerate}

Finally the result follows by [Corollary 11,\cite{K1}] 
since $K\# K$ is equivalent to its mirror image.
\end{proof}

\begin{thm}\label{homology}
The Kanenobu knot $K(p,0)$ for a negative integer $p$ is
homologically thin over $\mathbb{Q}$ and its Khovanov bigraded
homology groups are given as follows:

\[
\dim \mathcal{H}^{i,j}(K(p,0)) =
\begin{cases}
\dim \mathcal{H}^{-p,-2p-1}(K(0,0)) + 1, & \text{if} \ {(i,j) =
(0,1)},
      \\
    \dim \mathcal{H}^{-p,-2p+1}(K(0,0)) + 1,
     & \text{if} \ {(i,j) = (0,-1),}
     \\
\dim \mathcal{H}^{-p-1,-2p-3}(K(0,0)), & \text{if} \  {(i,j) = (-1,
-3)  \ \text{and}  \ p \neq -1 },
      \\
\dim \mathcal{H}^{-p-1,-2p-1}(K(0,0)), & \text{if} \  {(i,j) = (-1,
-1)  \ \text{and}  \ p \neq -1 },
      \\
\dim \mathcal{H}^{0,1}(K(0,0)) - 1, & \text{if} \  {(i,j) = (p, 2p +
1) },
\\
\dim \mathcal{H}^{0,-1}(K(0,0)) - 1, & \text{if} \  {(i,j) = (p, 2p
- 1) },
\\
\dim \mathcal{H}^{i + p, j + 2p}(K(0,0)), & \text{otherwise}.
     \\
\end{cases}
\]
\end{thm}

\begin{proof}

We show the claim by induction on $|p|$. We resolve any crossing of
the $p$-crossings to obtain $D(*0)$ and $D(*1)$. It is clear that
$D(*0)$ is a diagram of the Kanenobu knot $K(p + 1, 0)$ and $D(*1)$
is a diagram of the unlink of two components as seen in figure
\ref{figure14}.

\begin{figure}[h]
\centering
\includegraphics[width=10cm,height=8cm]{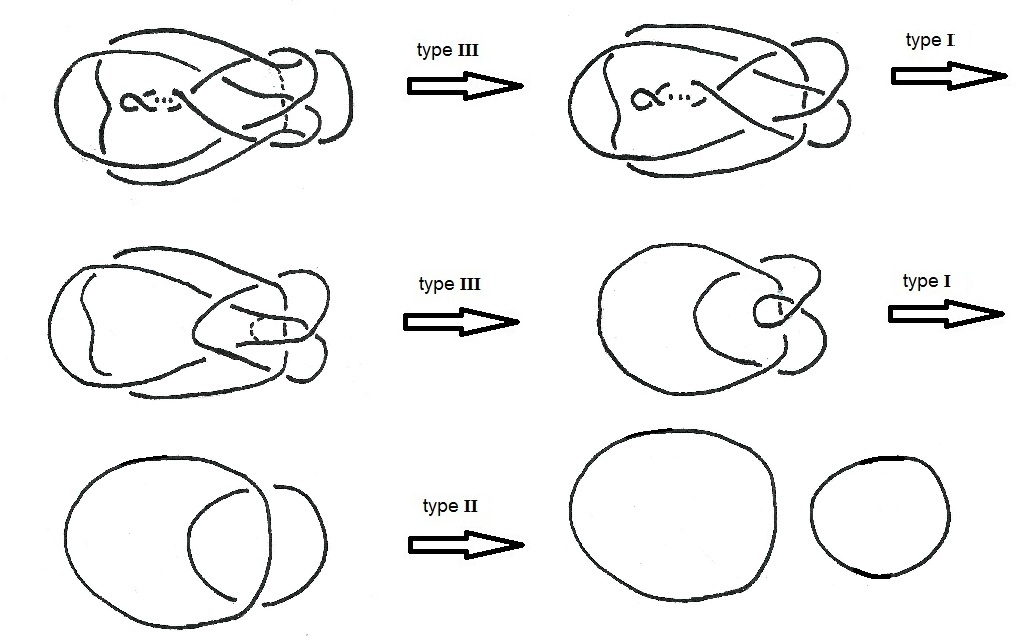}
\caption{Resolving one of the crossings of the diagram of $K(p,0)$.}
\label{figure14}
\end{figure}

From the induction hypothesis, $K(p + 1, 0)$ is homologically thin.
Therefore, $\mathcal{H}(K(p+1,0))$ is supported on the two lines $j
- 2i = \sigma(K(p + 1, 0)) \pm 1 =  \pm 1$. As $x(K(p+1,0)) =
x(D(*1)) = 3-p$ and $y(K(p+1,0)) = y(D(*1)) = 4$, then
$\overline{\mathcal{H}}(K(p+1,0))$ is supported on the two lines $j
- 2i = -4 \pm 1$. Also, we have

\[
\overline{\mathcal{H}}^{i, j}(D(*1)) =
\begin{cases}
\mathbb{Q}, & \text{if $(i, j) = (3-p, 4-2p)$, or $(3-p, -2p)$},
\\
\mathbb{Q}\oplus\mathbb{Q}, & \text{if $(i, j) = (3-p, 2- 2p)$},
\\
0, & \text{otherwise}.
\end{cases}
\]
Now we obtain $\dim \overline{\mathcal{H}}^{i, j}(K(p,0)) = \dim
\overline{\mathcal{H}}^{i, j}(K(p+1,0))$ directly from the long
exact sequence on homology except on the following three cases:
\begin{enumerate}
    \item
    \[
     0 \rightarrow \mathbb{Q} \rightarrow \overline{\mathcal{H}}^{4-p, 5- 2p}(K(p,0)) \rightarrow
     \overline{\mathcal{H}}^{4-p, 5 -2p}(K(p+1,0)) \rightarrow 0.
     \]
     Therefore, we obtain $\dim \overline{\mathcal{H}}^{4-p, 5-2p}(K(p,0))= \dim
     \overline{\mathcal{H}}^{4-p, 5-2p}(K(p+1,0)) + 1$.
      \item
     \[
     0 \rightarrow \overline{\mathcal{H}}^{3-p, 1-2p}(K(p,0))\rightarrow
     \overline{\mathcal{H}}^{3-p, 1-2p}(K(p+1,0)) \stackrel{\delta}{\rightarrow}
     \mathbb{Q} \rightarrow \overline{\mathcal{H}}^{4-p, 1-2p}(K(p,1)) \rightarrow 0.
     \]
     We have two subcases and only the second one holds:
\begin{enumerate}
    \item $\dim \overline{\mathcal{H}}^{3-p, 1-2p}(K(p,0)) = \dim
    \overline{\mathcal{H}}^{3-p, 1-2p}(K(p+1,0))$ and

    $\dim\overline{\mathcal{H}}^{4-p, 1-2p}(K(p,0)) = 1$.
    By induction,
     $K(p+1,0)$ is homologically thin.
    So by Lemma \ref{seq}, we obtain $\dim \overline{\mathcal{H}}^{3-p, 1-2p}(K(p+1,0)) =
    \dim \overline{\mathcal{H}}^{4-p, 5-2p}(K(p+1,0)) + 1$.
    Hence we conclude $ \dim \overline{\mathcal{H}}^{3-p, 1-2p}(K(p,0)) =
    \dim \overline{\mathcal{H}}^{4-p, 5-2p}(K(p,0))$.
    Now in the spectral sequence:
    \[  0 \rightarrow \overline{\mathcal{H}}^{3-p, 1-2p}(K(p,0)) \stackrel{d_{1}}{\rightarrow}
    \overline{\mathcal{H}}^{4-p, 5-2p}(K(p,0))\rightarrow 0.\]
    The map  $d_{1}$ has to
    be injective and not surjective since in the $E_{\infty}$-page one copy
    of $\mathbb{Q}$ survives at the $j$-grading  1 and not at the $j$-grading -3.
    This is impossible as the domain and the codomain have the same dimension.

    \item $\dim \overline{\mathcal{H}}^{3-p, 1-2p}(K(p,0)) = \dim
    \overline{\mathcal{H}}^{3-p, 1-2p}(K(p+1,0)) - 1,$ and

    $\dim
    \overline{\mathcal{H}}^{4-p, 1-2p}(D) = 0$.
\end{enumerate}

     \item
     \begin{align*}
    & 0 \rightarrow \overline{\mathcal{H}}^{3-p, 3-2p}(K(p,0))
    \rightarrow \overline{\mathcal{H}}^{3-p, 3-2p}(K(p+1,0))
    \stackrel{\delta}{\rightarrow} \mathbb{Q}\oplus \mathbb{Q} \\ & \rightarrow
    \overline{\mathcal{H}}^{4-p, 3-2p}(K(p,0)) \rightarrow
    \overline{\mathcal{H}}^{4-p, 3-2p}(K(p+1,0)) \rightarrow 0.
    \end{align*}
    We conclude that the Kanenobu knot $K(p,0)$ is homologically thin
    as the its bigraded homology is supported on $ j - 2i = -4 \pm 1$.
    Therefore as a result of the third statement of Lemma \ref{seq}, we obtain
    \begin{align*}
    \dim \overline{\mathcal{H}}^{3-p, 3-2p}(K(p,0)) + 1 & =
    \dim \overline{\mathcal{H}}^{2-p, -1-2p}(K(p,0)) + 1 \\
    & = \dim \overline{\mathcal{H}}^{2-p, -1-2p}(K(p+1,0)) + 1 \\
    & = \dim \overline{\mathcal{H}}^{3-p, 3-2p}(K(p+1,0)).
    \end{align*}

Also since the Euler characteristic of an exact sequence must
vanish, we obtain
\begin{align*}
\dim \overline{\mathcal{H}}^{4-p, 3-2p}(K(p,0)) = \dim
\overline{\mathcal{H}}^{4-p, 3-2p}(K(p+1,0)) + 1.
\end{align*}

\end{enumerate}
Finally, the recursion holds from the fact that $x(K(p,0)) = 4-p$
and $y(K(p,0)) = 4$ and equation \ref{relation}.
\end{proof}

\begin{minipage}{0.45\textwidth}

As a result of Theorem \ref{properties} and [Corollary
11,\cite{K1}], we obtain
\begin{coro}\label{image}
The Kanenobu knot $K(p,0)$ for a positive integer $p$ is
homologically thin over $\mathbb{Q}$ and its Khovanov bigraded
homology groups are given by $\mathcal{H}^{i,j}(K(p,0)) =
\mathcal{H}^{-i,-j}(K(-p,0))$.
\end{coro}

We show that $\mathcal{H}^{i,j}(K(p,q))$ depends only on the sum of
$p$ and $q$ and this generalizes [Theorem 7, \cite{GW}]. So we can
compute the rational Khovanov bigraded homology groups of any
Kanenobu knot by combining Theorem \ref{homology} and Corollary
\ref{image}.

\begin{prop}
For any two integers $p$ and $q$, we have
\[\mathcal{H}^{i,j}(K(p,q)) \equiv \mathcal{H}^{i,j}(K(p + q, 0)).\]
\end{prop}

\begin{proof}
We can assume that $p<0$ and $|p| > q$ using the second part of
Proposition \ref{properties}.  We have $x(K(p,q)) = x(K(p+q,0)),
y(K(p,q)) = y(K(p+q,0)), x(K(p+1,q)) = x(K(p+q+1,0)), y(K(p+1,q)) =
y(K(p+q+1,0)),$ and $x(U_{1}) = x(U_{2}), y(U_{1}) = y(U_{2}),$
where $U_{1}$ and $U_{2}$ are diagrams of the unlink of two
components obtained by resolving one of the $p$-th and $p+q$-th
crossings in the diagrams of $K(p,q)$ and $K(p+q,0)$ respectively.
Therefore, it is enough to show that
$\overline{\mathcal{H}}^{i,j}(K(p,q)) \equiv
\overline{\mathcal{H}}^{i,j}(K(p + q, 0))$.

The long exact sequences for $K(p,q)$ and $K(p + q, 0)$ as in
equation \ref{sequence} are isomorphic as seen in the following
commutative diagram using the five lemma if we use induction on
$|p+q|$.
\end{proof}
\end{minipage}%
\hfill
\begin{minipage}{0.45\textwidth}
\begin{turn}{-90}
$\begin{CD} \overline{\mathcal{H}}^{i-1,j}(K(p+1,q)) @>>>
\overline{\mathcal{H}}^{i-1,j-1}(U_{1}) @>>>
\overline{\mathcal{H}}^{i,j}(K(p,q)) @>>>
\overline{\mathcal{H}}^{i,j}(K(p+1,q))@>>> \overline{\mathcal{H}}^{i,j-1}(U_{1})\\
 @|        @|     @VVV     @|           @|\\
\overline{\mathcal{H}}^{i-1,j}(K(p+q+1,0))@>>>
\overline{\mathcal{H}}^{i-1,j-1}(U_{2}) @>>>
\overline{\mathcal{H}}^{i,j}(K(p+q,0)) @>>>
\overline{\mathcal{H}}^{i,j}(K(p+q+1,0))@>>>
\overline{\mathcal{H}}^{i,j-1}(U_{2})
\end{CD}$
\end{turn}
\end{minipage}%








\end{document}